\documentclass[11pt]{article}
\usepackage[T1]{fontenc}
\usepackage[utf8]{inputenc}
\usepackage{lmodern}
\usepackage{amsmath,amssymb,amsthm,mathtools}
\usepackage[margin=2.4cm]{geometry}
\usepackage{booktabs}
\usepackage{parskip}           
\usepackage[backref=page]{hyperref}
\hypersetup{colorlinks=true,linkcolor=blue,urlcolor=blue}

\hypersetup{hidelinks}

\newtheorem{theorem}{Theorem}[section]
\numberwithin{equation}{section}
\newtheorem{proposition}[theorem]{Proposition}
\newtheorem{lemma}[theorem]{Lemma}
\newtheorem{corollary}[theorem]{Corollary}
\newtheorem{introthm}{Theorem}            

\theoremstyle{definition}

\newtheorem{remark}[theorem]{Remark}
\newtheorem{example}[theorem]{Example}

\newcommand{\YB}{\mathcal{YB}}

\newcommand{\UG}{\mathcal{UG}}

\usepackage{enumerate}          
\newcommand{\selabel}[1]{\label{se:#1}}

\newcommand{\sselabel}[1]{\label{sse:#1}}
\newcommand{\sseref}[1]{Subsection~\ref{sse:#1}}
\def\thlabel#1{\label{th:#1}}
\def\thref#1{Theorem~\ref{th:#1}}

\newcommand{\lelabel}[1]{\label{le:#1}}
\newcommand{\leref}[1]{Lemma~\ref{le:#1}}
\newcommand{\prlabel}[1]{\label{pr:#1}}
\newcommand{\prref}[1]{Proposition~\ref{pr:#1}}
\newcommand{\colabel}[1]{\label{co:#1}}
\newcommand{\coref}[1]{Corollary~\ref{co:#1}}
\newcommand{\relabel}[1]{\label{re:#1}}
\newcommand{\reref}[1]{Remark~\ref{re:#1}}
\newcommand{\exlabel}[1]{\label{ex:#1}}
\newcommand{\exref}[1]{Example~\ref{ex:#1}}

\newcommand{\eqlabel}[1]{\label{eq:#1}}
\newcommand{\equref}[1]{(\ref{eq:#1})}

\newcommand{\bZ}{\mathbb{Z}}

\title{Every group is the automorphism group of a set-theoretic solution of the Yang--Baxter equation}
\author{G. Militaru}
\date{}

\begin{document}
\maketitle

\begin{abstract}
We give an affirmative answer to the \emph{abstract group problem} for the category $\YB$ of all set-theoretic solutions of the Yang--Baxter (YB) equation: every group is isomorphic to the automorphism group of a YB-solution, which can be chosen on a finite set if the group is finite. In the terminology of L\'aszl\'o Babai, this means that the category $\YB$ is \emph{universal} and \emph{finitely universal}, which is an analogue of Frucht's classical theorem for graphs. Furthermore, the YB-solutions can be chosen within three very small classes: the solutions of the Frobenius-Separability equation, the solutions of the Baaj--Long--Skandalis equation and the recently introduced bipartite solutions. The proof is based on new constructions which associate to any graph a YB-solution in each of these classes, with the same automorphism group as the graph. Explicit YB-solutions with automorphism groups $C_n$, $S_n$, $D_n$ and $Q_8$ are given, as well as a unitary YB-solution, constructed directly from a field $K$, whose automorphism group is ${\rm Aut} (K)$.
\end{abstract}

\medskip


\vspace{.3cm}



\section*{Introduction}

The \emph{representation problem}, one of the classical themes of algebra and combinatorics, asks which abstract groups occur as the full automorphism group of an object $\mathcal{O}$ of a given category $\mathcal{C}$. In this case we say that the object $\mathcal{O}$ \emph{represents} the given group \cite{Babai1981}. Following Babai \cite[Section 4.1]{Babai1995}, the category $\mathcal{C}$ is called \emph{universal} (respectively, \emph{finitely universal}) if every group (respectively, every finite group) is represented by an object (respectively, by a ``finite'' object) of $\mathcal{C}$. The problem was formulated by K\H{o}nig \cite{Kon} ninety years ago for the category of graphs, and it was answered affirmatively by Frucht \cite{Frucht39}: every finite group is isomorphic to the automorphism group of a finite graph. This remarkable theorem was later extended to arbitrary groups, possibly infinite, independently by de Groot \cite{deGroot59} and Sabidussi \cite{Sabidussi60}.

The representation problem is not specific to graph theory. Frucht later studied the same problem for posets and lattices \cite{Frucht1950poset, Frucht1950lattice}. These early results made it clear that the question is a general one: to represent abstract groups as automorphism groups of objects of a prescribed category. Since then, the problem has been studied continuously, its systematic study being closely associated with Babai, whose surveys \cite{Babai1981, Babai1995} have become standard references in the field. Babai also emphasized \emph{quantitative} questions: once the existence of an object $X$ with ${\rm Aut} (X) \cong G$ is known, how ``small'' can $X$ be, and how restricted can the class of admissible objects be? In this direction he obtained substantial improvements over the early constructions.

A categorical extension of the subject was developed by Hedrl\'in, Pultr, Trnkov\'a, Vop\v{e}nka and their collaborators; the standard reference is the monograph \cite{PultrTrnkova80}. There the question above appears as the \emph{abstract group problem}, together with the \emph{abstract monoid problem}, in which automorphism groups are replaced by endomorphism monoids \cite[Introduction, Section 2]{PultrTrnkova80}.

The representation problem has also been studied for algebraic objects. Gordeev and Popov \cite{GordeevPopov2003} proved that, over suitable fields, every linear algebraic group is the automorphism group of a finite-dimensional simple algebra, not necessarily associative. More recently, a quantum version of Frucht's theorem was proved: every finite quantum group is the quantum automorphism group of a finite quantum graph \cite{BGMSW25}. The subject remains very active today, as it has been throughout its history (see, for instance, \cite{Barbieri25, BarmakBarreto24, BasicFowler25, Cuypers26, MorrisSpiga2025} and the references therein).

The quantum Yang--Baxter equation first appeared in theoretical physics, in the work of Yang \cite{Yang} on many-body problems and of Baxter \cite{baxter} on exactly solvable models of statistical mechanics. It later became a central object in the theory of quantum groups and Hopf algebras, and it is closely related to braid groups, invariants of knots and braided monoidal categories. Drinfel'd \cite{dr} proposed the study of its set-theoretic solutions, i.e. of the maps $R \colon X \times X \to X \times X$, where $X$ is a set, which satisfy the Yang--Baxter equation. This problem has a strong combinatorial flavour, and over the last thirty years it has been studied intensively; for a detailed account we refer to the monograph \cite{BES}.

Our starting point is \cite[Remark 4.5]{acm1}, where it was asked which groups are of the form ${\rm Aut} (X, \, R)$, for some solution $R \colon X \times X \to X \times X$ of the set-theoretic Frobenius-Separability equation $R^{12}R^{23} = R^{23}R^{13} = R^{13}R^{12}$, a special class of solutions of the braid equation. We had in mind, of course, the classical results on the representation problem and the following observation of Babai \cite{Babai1995}:``Take almost any \emph{reasonably broad} class of combinatorial or algebraic structures; the class will be universal.'' The words ``reasonably broad'' are essential here: the category of groups is not universal, since no group has the cyclic group of order $3$ as its automorphism group. Moreover, a certain class of FS-solutions is neither universal nor finitely universal: the automorphism groups of these solutions are exactly the centralizers of single maps, and not every finite group is of this form \cite[Theorems 2.19 and 4.2(2)]{acm1}. Recently, it was shown in \cite{MilRacks} that the category of racks, which can be viewed as a full subcategory of the category of all YB--solutions via the derived solution of a rack \cite[Remark 6.2]{BES}, is neither universal nor finitely universal: the quaternion group $Q_8$ is not the automorphism group of any rack \cite[Theorem A]{MilRacks}. Therefore, deciding whether a given category, or class of objects, is (finitely) universal remains an interesting problem.

Graphs have already been used in the study of set-theoretic YB-solutions, mostly in the opposite direction: already in \cite[Section 1.2]{etingof}, solutions are used to color the edges of $4$-valent flat graphs; in \cite{GIM}, a labelled graph is associated to a solution and used to compute its automorphism group, while in \cite{Yang16} the solutions are related to higher-rank graphs. At the categorical level, this direction is well understood: the category $\UG$ of undirected graphs is alg-universal \cite[Chapter IV, Theorem 3.1]{PultrTrnkova80}; hence the category $\YB$ of all YB-solutions, which is a variety of algebras \cite{CMv1} and in particular
an algebraic category in the sense of \cite[Chapter I, 8.4]{PultrTrnkova80}, admits a full embedding, i.e. a full and faithful functor which is injective on objects,
\begin{equation}\eqlabel{YBinUG}
\iota \colon \YB \hookrightarrow \UG.
\end{equation}
Roughly speaking, \equref{YBinUG} allows us to regard $\YB$ as a full subcategory of $\UG$; this suggests that the tools of graph theory may also be useful in the study of the Yang--Baxter equation. 
In this direction, it was shown in \cite[Theorem 2.19]{acm1} that a restricted class of FS-solutions is parametrized by functional digraphs (essentially those of Harary \cite{harary}), and that the automorphism groups of these solutions are the automorphism groups of the corresponding digraphs. To the best of our knowledge, solutions of the Yang--Baxter equation associated to arbitrary graphs have not been considered before. At a more general, categorical level, we are interested in the existence of a functor $F \colon \UG \to \YB$ which is full and essentially surjective, or even a full retraction of the embedding \equref{YBinUG}, i.e. such that $F \circ \iota \cong {\rm Id}_{\YB}$.

In this paper we address the representation problem for the category of all set-theoretic solutions of the Yang--Baxter equation (respectively, of the braid equation), and we show that the answer is positive, already inside three very small classes of solutions (see \cite[Section 3]{MilBIP}). The strategy is a natural one: to any graph $\Gamma = (V, \, E)$ with at least one edge we associate
\begin{itemize}
\item an FS-solution $R_\Gamma$, and hence a solution of the braid equation, on the set $(V \times E) \sqcup \{x_{v, e} \, | \, v \in e\}$ (\sseref{FS});
\item a solution $R'_\Gamma$ of the Baaj--Long--Skandalis equation, and hence of the Yang--Baxter equation, on a set with $2 |V|$ elements (\sseref{BLS});
\item a bipartite solution $R''_\Gamma$ in the sense of \cite{MilBIP}, in particular an involutive solution of the braid equation, on the set $E \sqcup V$ (\sseref{bij}).
\end{itemize}
In all three cases the automorphism group of the solution is isomorphic to ${\rm Aut} (\Gamma)$ (\thref{autFS}, \thref{autBLS} and \thref{autbij}). Combined with the theorems of Frucht, de Groot and Sabidussi,
this gives \thref{mainth}:

\begin{introthm}
Any group $G$ is isomorphic to the automorphism group of a set-theoretic solution of the Yang--Baxter equation, and to the automorphism group of a set-theoretic solution of the braid equation. If $G$ is finite, then the solutions can be chosen on finite sets. Moreover:
\begin{enumerate}[(1)]
\item the solution of the braid equation can be chosen to be a solution of the Frobenius-Separability equation;
\item the solution of the Yang--Baxter equation can be chosen to be an idempotent solution of the Baaj--Long--Skandalis equation;
\item the solution of the braid equation can be chosen to be a bipartite solution, in particular an involutive one, and the solution of the Yang--Baxter equation can be chosen to be unitary.
\end{enumerate}
\end{introthm}

In the terminology of Babai recalled above, the theorem says that the classes of FS-solutions, of idempotent BLS-solutions and of bipartite solutions are universal and finitely universal. In \sseref{examples} we give explicit solutions of the Yang--Baxter equation whose automorphism groups are the cyclic group $C_n$, the symmetric group $S_n$, the dihedral group $D_n$ of order $2n$ and the quaternion group $Q_8$. Finally, we construct unitary YB-solutions directly from fields (\thref{field}):

\begin{introthm}
For any field $K$ there exists a unitary YB-solution $(X, \, r)$ such that ${\rm Aut} (X, \, r) \cong {\rm Aut} (K)$.
\end{introthm}

Combined with a theorem of Fried and Koll\'ar (\cite{FK77, FK78} and \cite[Theorem 4.2]{Babai1995}), which states that every group is isomorphic to the automorphism group of a field, Theorem B gives a second proof of the fact that any group is isomorphic to the automorphism group of a unitary YB-solution.

\section{Preliminaries}\selabel{prel}

Given a set $X$, we denote by $\tau_X \colon X\times X \to X \times X$ the flip map, $\tau_X (x, \, y) = (y, \, x)$; ${\rm Sym} (X)$ denotes the symmetric group on $X$, i.e. the group of all
bijections $X \to X$. For a map $R\colon X \times X \to X\times X$ we use the usual notation $R^{12} := R \times {\rm Id}_X$, $R^{23} := {\rm Id}_X \times R$ and $R^{13}: = ({\rm Id}_X \times \tau_X) \circ R^{12} \circ ({\rm Id}_X \times \tau_X)$. The map $R$ is a (set-theoretic) solution of the \emph{Yang--Baxter equation}, or a \emph{YB-solution}, if $R^{12} R^{13} R^{23} = R^{23}R^{13}R^{12}$, and a solution of the \emph{braid equation} if $R^{12}R^{23}R^{12} = R^{23}R^{12}R^{23}$. It is well known that $R$ is a solution of the braid equation if and only if $\tau_X \circ R$ is a YB-solution. A map $R\colon X \times X \to X\times X$ is called
\emph{unitary} if $R^{21} R = {\rm Id}_{X \times X}$, where $R^{21} := \tau_X R \tau_X$. Any map $R \colon X \times X \to X \times X$ can be written as
\begin{equation}\eqlabel{notatieR}
R(x, \, y) = (x \cdot y, \, x \ast y),
\end{equation}
where $(X, \, \cdot, \, \ast)$ is a \emph{bi-magma}, i.e. a set with two binary operations; this gives a bijection between maps $R \colon X \times X \to X \times X$ and bi-magma structures on $X$. In these terms, $R$ is a YB-solution if and only if for all $x$, $y$, $z\in X$:
\begin{align}
  (x\cdot y) \cdot z &= \bigl(x \cdot (y \ast z) \bigr) \cdot (y \cdot z)\eqlabel{YB1}\\
  \bigl( x \cdot (y \ast z) \bigr) \ast (y\cdot z ) &= (x \ast y) \cdot \bigl( (x\cdot y) \ast z  \bigr)\eqlabel{YB2} \\
  x \ast (y \ast z) &=  (x \ast y) \ast \bigl( (x\cdot y) \ast z\bigr).\eqlabel{YB3}
\end{align}
A \emph{morphism} of solutions $\sigma \colon (X, \, R) \to (X', \, R')$ is a map $\sigma \colon X \to X'$ such that $(\sigma \times \sigma) \circ R = R' \circ (\sigma \times \sigma)$, i.e. a morphism of the associated bi-magmas: $\sigma (x \cdot y) = \sigma(x) \cdot \sigma(y)$ and $\sigma (x \ast y) = \sigma(x) \ast \sigma(y)$. We denote by $\YB$ the category of all YB-solutions, i.e. its objects are pairs $(X, \, R)$ consisting of a set $X$ and a YB-solution $R: X \times X \to X \times X$. The \emph{automorphism group} ${\rm Aut} (X, \, R)$ of $(X, \, R)$ is the group of all bijections $\sigma \colon X \to X$ with $(\sigma \times \sigma) \circ R = R \circ (\sigma \times \sigma)$, i.e. the automorphism group of the bi-magma $(X, \, \cdot, \, \ast)$. Since $(\sigma \times \sigma) \circ \tau_X = \tau_X \circ (\sigma \times \sigma)$ for any map $\sigma$, we have
\begin{equation}\eqlabel{auttau}
{\rm Aut} (X, \, R) = {\rm Aut} (X, \, \tau_X \circ R).
\end{equation}
Consequently, the classes of groups which are automorphism groups of solutions of the braid equation and of solutions of the Yang--Baxter equation coincide. In particular, a class of solutions of the braid equation, such as the FS-solutions or the bipartite solutions considered below, will be regarded as a class of YB-solutions via $R \mapsto \tau_X \circ R$; by \equref{auttau}, this does not change the automorphism groups.

We recall the few notions of graph theory that we need. A \emph{graph} $\Gamma = (V, \, E)$ consists of a set $V$ of \emph{vertices} and a set $E$ of $2$-element subsets of $V$, called \emph{edges}; thus our graphs are undirected, without loops and without multiple edges. Two vertices $u \neq v$ are \emph{adjacent}, and we write $u \sim v$, if $\{u, \, v\} \in E$; $u$ and $v$ are the \emph{endpoints} of the edge $\{u, \, v\}$, and a vertex $v$ and an edge $e$ are \emph{incident} if $v \in e$. An \emph{automorphism} of $\Gamma$ is a bijection $\pi \colon V \to V$ such that, for all $u$, $v\in V$, $\{u, \, v\} \in E$ if and only if $\{\pi(u), \, \pi(v)\} \in E$; the automorphisms of $\Gamma$ form a group ${\rm Aut} (\Gamma)$. Since they will be used intensively in what follows, we recall the following results: the first one is due to Frucht \cite{Frucht39}, while the second one is due to de Groot \cite{deGroot59} and Sabidussi \cite{Sabidussi60}.

\begin{theorem}\thlabel{Frucht}
Any finite group $G$ is isomorphic to ${\rm Aut} (\Gamma)$, for some finite graph $\Gamma$. Any group $G$ is isomorphic to ${\rm Aut} (\Gamma)$, for some graph $\Gamma$.
\end{theorem}

A graph $\Gamma = (V, \, E)$ and its \emph{complement} $\Gamma^c = (V, \, E^c)$, where $E^c$ is the set of all $2$-element subsets of $V$ which are not in $E$, have the same automorphism group. Hence, in \thref{Frucht}, if $|V| \geq 2$ we can assume that $E \neq \emptyset$, replacing $\Gamma$ by $\Gamma^c$ if necessary; if $|V| \leq 1$, then $G$ is trivial.

\begin{remark}\relabel{bimagma}
Since ${\rm Aut} (X, \, R)$ is the automorphism group of the bi-magma $(X, \, \cdot, \, \ast)$ associated to $R$ by \equref{notatieR}, it is not surprising that any group is the automorphism group of a bi-magma; in fact, one binary operation is enough. Indeed, for a graph $\Gamma = (V, \, E)$ let $x \cdot y := y$ if $\{x, \, y\} \in E$ and $x \cdot y := x$ otherwise. A bijection $\sigma$ of $V$ preserves the operation $\cdot$ if and only if, for all $x \neq y$, we have $\{x, \, y\} \in E \Leftrightarrow \{\sigma(x), \, \sigma(y)\} \in E$; thus ${\rm Aut} (V, \, \cdot) = {\rm Aut} (\Gamma)$, and \thref{Frucht} shows that any group is the automorphism group of a magma, which is finite if the group is finite. The point of the main theorem is therefore that the identities \equref{YB1}-\equref{YB3}, as well as the stronger conditions defining the FS-solutions, the idempotent BLS-solutions or the involutive solutions of the braid equation, impose no restriction on the automorphism group. This is in contrast with the class of FS-solutions whose left translations $l_x$ are bijective: their automorphism groups are the centralizers of single maps \cite[Theorem 2.19]{acm1}, and this class is not finitely universal \cite[Theorem 4.2(2)]{acm1}.
\end{remark}

\section{Frucht-type theorems for YB-solutions}\selabel{frucht}
In order to show that the categories of FS-solutions, BLS-solutions and bipartite solutions are universal and finitely universal, it is enough to associate to any graph with at least one edge an FS-solution, a BLS-solution and a bipartite solution, respectively, whose automorphism groups are isomorphic to the automorphism group of the graph; the conclusion then follows from \thref{Frucht}.

\subsection{FS-solutions}\sselabel{FS}

We recall that a map $R \colon X \times X \to X \times X$ is a solution of the \emph{Frobenius-Separability equation}, or an \emph{FS-solution}, if $R^{12}R^{23} = R^{23}R^{13} = R^{13}R^{12}$. Any FS-solution is a solution of the braid equation, since $R^{12}R^{23}R^{12} = R^{23}R^{13}R^{12} = R^{23}R^{12}R^{23}$. By \cite[Proposition 2.4]{acm1}, a map $R$, written as in \equref{notatieR}, is an FS-solution if and only if for all $x$, $y$, $z\in X$:
\begin{align}
x \cdot (y \cdot z) &= (x \cdot y) \cdot z = x \cdot z \eqlabel{fs1}\\
x \ast (y \cdot z) &= y \cdot (x \ast z) = x \ast y \eqlabel{fs2}\\
(x \cdot y) \ast z &= y \ast (x \ast z) = y \ast z. \eqlabel{fs3}
\end{align}

In \cite[Section 4]{acm1} we described a large class of groups which arise as automorphism groups of FS-solutions; in particular, any finite cyclic group belongs to this class.

\begin{example}\exlabel{autP3}
Let $n \geq 2$ and let $f \colon \bZ_n \to \bZ_n$, $f(z) := z + 1$, be the cyclic permutation of $\bZ_n = \{0, \, 1, \ldots, n-1\}$. By \cite[Example 2.3(3)]{acm1}, the map
\begin{equation*}
R_f \colon \bZ_n \times \bZ_n \to \bZ_n \times \bZ_n, \qquad R_f (x, \, y) := \bigl( y, \, f(x) \bigr) = (y, \, x + 1),
\end{equation*}
is a bijective FS-solution \cite[Theorem 2.19]{acm1}. Moreover, ${\rm Aut} (\bZ_n, \, R_f)$ is the centralizer of $f$ in ${\rm Sym} (\bZ_n)$ \cite[Theorem 2.19(3)]{acm1}.
If $\sigma \circ f = f \circ \sigma$, then $\sigma (z) = \sigma (f^z (0)) = f^z (\sigma (0)) = z + \sigma(0)$, for all $z \in \bZ_n$; hence
\begin{equation*}
{\rm Aut} (\bZ_n, \, R_f) = \{ \sigma_v \, | \, v\in \bZ_n \} \cong C_n, \qquad \text{where} \,\,\, \sigma_v (z) := z + v.
\end{equation*}
Thus any finite cyclic group $C_n$ is the automorphism group of a bijective FS-solution on a set with $n$ elements. The group $C_n$ can also be obtained from \thref{autFS} below, applied to a graph whose automorphism group is $C_n$, but the resulting solution lives on a much larger set.
\end{example}

The FS-solutions we need are of a very simple type: in the
notation of \cite[Theorem 2.10]{acm1}, they correspond to a rectangular band $A \times B$ and $\vartheta = {\rm Id}_B$, with some additional elements.

Let $\Gamma = (V, \, E)$ be a graph with $E \neq \emptyset$. We consider the set
\begin{equation*}
X_\Gamma := (V \times E) \, \sqcup \, \{x_{v, e} \, | \, v \in V, \, e \in E, \, v \in e\},
\end{equation*}
i.e. the disjoint union of $V \times E$ with a set containing one new element $x_{v, e}$ for each incident pair (vertex, edge); thus each edge $e = \{u, \, v\}$ produces the two new elements $x_{u, e}$ and $x_{v, e}$, and $|X_\Gamma| = |V| \, |E| + 2 |E|$ if $\Gamma$ is finite. Let $a \colon X_\Gamma \to V$ and $b \colon X_\Gamma \to E$ be the maps
\begin{equation*}
a(v, \, e) = a(x_{v, e}) := v, \qquad b(v, \, e) = b(x_{v, e}) := e.
\end{equation*}
Both $a$ and $b$ are surjective, since $V \times E \subseteq X_\Gamma$ and $V \neq \emptyset \neq E$. We define a bi-magma structure on $X_\Gamma$ by
\begin{equation}\eqlabel{RGamma}
x \cdot y := \bigl( a(x), \, b(y) \bigr), \quad x \ast y := \bigl( a(y), \, b(x) \bigr), \quad {\rm i.e.} \quad R_\Gamma (x, \, y) = \Bigl( \bigl( a(x), \, b(y) \bigr), \, \bigl( a(y), \, b(x) \bigr) \Bigr),
\end{equation}
for all $x$, $y\in X_\Gamma$.

\begin{lemma}\lelabel{RGammaFS}
Let $\Gamma = (V, \, E)$ be a graph with $E \neq \emptyset$. Then the map $R_\Gamma \colon X_\Gamma  \times X_\Gamma  \to X_\Gamma  \times X_\Gamma $  given by \equref{RGamma} is an FS-solution,
and hence a solution of the braid equation; moreover, $R_\Gamma^3 = R_\Gamma$.
\end{lemma}

\begin{proof}
By \equref{RGamma}, for all $x$, $y\in X_\Gamma$ we have
\begin{equation}\eqlabel{abprod}
a(x \cdot y) = a(x), \qquad b(x \cdot y) = b(y), \qquad a(x \ast y) = a(y), \qquad b(x \ast y) = b(x).
\end{equation}
Therefore, for all $x$, $y$, $z\in X_\Gamma$:
\begin{align*}
x \cdot (y \cdot z) &= \bigl( a(x), \, b(z) \bigr) = (x \cdot y) \cdot z = x \cdot z,\\
x \ast (y \cdot z) &= \bigl( a(y), \, b(x) \bigr) = y \cdot (x \ast z) = x \ast y,\\
(x \cdot y) \ast z &= \bigl( a(z), \, b(y) \bigr) = y \ast (x \ast z) = y \ast z,
\end{align*}
i.e. \equref{fs1}-\equref{fs3} hold. Moreover,
\[
R_\Gamma^2 (x, \, y) = R_\Gamma \bigl( (a(x), \, b(y)), \, (a(y), \, b(x)) \bigr) = \bigl( (a(x), \, b(x)), \, (a(y), \, b(y)) \bigr),
\]
and applying $R_\Gamma$ once more we obtain $R_\Gamma (x, \, y)$ again.
\end{proof}

The following lemma shows that any automorphism of the above FS-solution $(X_\Gamma, \, R_\Gamma)$ is induced by a pair of permutations of $V$ and $E$ preserving the incidence.

\begin{lemma}\lelabel{autFSstep}
Let $\sigma \in {\rm Aut} (X_\Gamma, \, R_\Gamma)$ be an automorphism of the FS-solution $(X_\Gamma, \, R_\Gamma)$ associated to a graph $\Gamma = (V, \, E)$. Then:
\begin{enumerate}[(1)]
\item there exist unique bijections $\sigma_V \colon V \to V$ and $\sigma_E \colon E \to E$ such that $a \circ \sigma = \sigma_V \circ a$, $b \circ \sigma = \sigma_E \circ b$ and
\[
\sigma (v, \, e) = (\sigma_V (v), \, \sigma_E (e)),
\]
for all $(v, \, e) \in V \times E$;
\item for all $v\in V$ and $e \in E$: $v \in e$ if and only if $\sigma_V (v) \in \sigma_E (e)$, and in this case $\sigma (x_{v, e}) = x_{\sigma_V (v), \, \sigma_E (e)}$.
\end{enumerate}
\end{lemma}

\begin{proof}
(1) Consider the equivalence relations on $X_\Gamma$ defined by $x \sim_1 x'$ if and only if $x \cdot u = x' \cdot u$ for all $u\in X_\Gamma$, and $y \sim_2 y'$ if and only if $u \cdot y = u \cdot y'$ for all $u\in X_\Gamma$. Since $x \cdot u = (a(x), \, b(u))$ and $b$ is surjective, $x \sim_1 x'$ if and only if $a(x) = a(x')$; similarly, $y \sim_2 y'$ if and only if $b(y) = b(y')$. If $x \sim_1 x'$, then $\sigma(x) \cdot \sigma(u) = \sigma (x \cdot u) = \sigma (x' \cdot u) = \sigma(x') \cdot \sigma(u)$ for all $u$; since $\sigma$ is surjective, $\sigma(x) \sim_1 \sigma(x')$. The same holds for $\sim_2$ and for $\sigma^{-1}$. Thus $\sigma$ induces bijections between the classes of $\sim_1$, i.e. a bijection $\sigma_V$ of $V$ with $a \circ \sigma = \sigma_V \circ a$ (using that $a$ is surjective), and similarly a bijection $\sigma_E$ of $E$ with $b \circ \sigma = \sigma_E \circ b$. Uniqueness follows from the surjectivity of $a$ and $b$. Finally, for $(v, \, e) \in V \times E$ we have $(v, \, e) = (v, \, e) \cdot (v, \, e)$ by \equref{RGamma}, hence, using \equref{RGamma},
\begin{equation*}
\sigma (v, \, e) = \sigma (v, \, e) \cdot \sigma (v, \, e) = \Bigl( a \bigl( \sigma (v, \, e) \bigr), \, b \bigl( \sigma (v, \, e) \bigr) \Bigr) = \bigl( \sigma_V (v), \, \sigma_E (e) \bigr).
\end{equation*}

(2) By \equref{RGamma}, the set of all products $X_\Gamma \cdot X_\Gamma$ is exactly $V \times E$, and any automorphism maps products to products; hence $\sigma (V \times E) = V \times E$ (by (1)), and therefore $\sigma$ maps the set of new elements $\{x_{v, e} \, | \, v \in e\}$ bijectively onto itself. If $v \in e$, then $\sigma (x_{v, e})$ is a new element with $a (\sigma (x_{v, e})) = \sigma_V (v)$ and $b (\sigma (x_{v, e})) = \sigma_E (e)$, by (1). Since the only new element $x_{u, f}$ with $a(x_{u, f}) = u$ and $b(x_{u, f}) = f$ is $x_{u, f}$ itself, and it exists only if $u \in f$, we obtain $\sigma_V (v) \in \sigma_E (e)$ and $\sigma (x_{v, e}) = x_{\sigma_V (v), \, \sigma_E (e)}$. Applying the same argument to $\sigma^{-1}$, whose associated bijections are $\sigma_V^{-1}$ and $\sigma_E^{-1}$, we obtain the converse implication.
\end{proof}

\begin{lemma}\lelabel{incidence}
Let $\Gamma = (V, \, E)$ be a graph and let $\sigma_V \colon V \to V$, $\sigma_E \colon E \to E$ be bijections such that $v \in e$ if and only if $\sigma_V (v) \in \sigma_E (e)$, for all $v\in V$ and $e\in E$. Then $\sigma_V \in {\rm Aut} (\Gamma)$ and $\sigma_E (\{u, \, v\}) = \{\sigma_V (u), \, \sigma_V (v)\}$ for all $\{u, \, v\} \in E$. Conversely, for any $\pi \in {\rm Aut} (\Gamma)$, the pair $(\pi, \, \pi_E)$, where $\pi_E (\{u, \, v\}) := \{\pi(u), \, \pi(v)\}$, satisfies this condition.
\end{lemma}

\begin{proof}
Let $e = \{u, \, v\} \in E$. The vertices incident with $\sigma_E (e)$ are exactly the images through $\sigma_V$ of the vertices incident with $e$; since an edge is the set of its two endpoints, $\sigma_E (e) = \{\sigma_V (u), \, \sigma_V (v)\}$. In particular, $\{\sigma_V (u), \, \sigma_V (v)\} \in E$ whenever $\{u, \, v\} \in E$. Applying this to the pair $(\sigma_V^{-1}, \, \sigma_E^{-1})$, which satisfies the same condition, we obtain the converse implication; hence $\sigma_V \in {\rm Aut} (\Gamma)$. The last statement is obvious: if $\pi \in {\rm Aut} (\Gamma)$, then $\pi_E$ is a bijection of $E$ (with inverse $(\pi^{-1})_E$), and $v \in e$ if and only if $\pi(v) \in \pi_E (e)$.
\end{proof}

With these three results at hand, we can now show that the category of all FS-solutions is universal and finitely universal.

\begin{theorem}\thlabel{autFS}
Let $\Gamma = (V, \, E)$ be a graph with $E \neq \emptyset$. Then the map
\begin{equation*}
{\rm Aut} (\Gamma) \to {\rm Aut} (X_\Gamma, \, R_\Gamma), \qquad \pi \mapsto \widehat{\pi}, \qquad \widehat{\pi} (v, \, e) := \bigl( \pi(v), \, \pi_E (e) \bigr), \quad \widehat{\pi} (x_{v, e}) := x_{\pi(v), \pi_E (e)},
\end{equation*}
is an isomorphism of groups. In particular, any group is isomorphic to the automorphism group of an FS-solution, and hence of a solution of the braid equation.\footnote{By \equref{auttau}, the YB-solution
$\tau_{X_\Gamma} \circ R_\Gamma$ has the same automorphism group ${\rm Aut} (\Gamma)$.} If the group is finite, the solution can be chosen on a finite set.
\end{theorem}

\begin{proof}
Let $\pi \in {\rm Aut} (\Gamma)$. By \leref{incidence}, $\widehat{\pi}$ is a well defined bijection of $X_\Gamma$ (with inverse $\widehat{\pi^{-1}}$), and $a \circ \widehat{\pi} = \pi \circ a$, $b \circ \widehat{\pi} = \pi_E \circ b$. Hence, by \equref{RGamma},
\begin{equation*}
\widehat{\pi} (x \cdot y) = \bigl( \pi (a(x)), \, \pi_E (b(y)) \bigr) = \bigl( a(\widehat{\pi}(x)), \, b(\widehat{\pi}(y)) \bigr) = \widehat{\pi}(x) \cdot \widehat{\pi}(y),
\end{equation*}
and similarly $\widehat{\pi} (x \ast y) = \widehat{\pi}(x) \ast \widehat{\pi}(y)$. Thus $\widehat{\pi} \in {\rm Aut} (X_\Gamma, \, R_\Gamma)$, and $\pi \mapsto \widehat{\pi}$ is clearly a morphism of groups; it is injective, since $\widehat{\pi} (v, \, e) = (\pi(v), \, \pi_E (e))$ for all $(v, \, e) \in V \times E$ and $E \neq \emptyset$, so that $\widehat{\pi}$ determines $\pi$.

Conversely, let $\sigma \in {\rm Aut} (X_\Gamma, \, R_\Gamma)$ and let $\sigma_V$, $\sigma_E$ be the bijections of \leref{autFSstep}. By \leref{autFSstep}(2) and \leref{incidence}, $\pi := \sigma_V \in {\rm Aut} (\Gamma)$ and $\sigma_E = \pi_E$, and then \leref{autFSstep} shows that $\sigma = \widehat{\pi}$. This proves the first statement.

Let now $G$ be a group. If $G$ is trivial, we can take the one-element solution. Otherwise, by \thref{Frucht} and the remark following it, there exists a graph $\Gamma = (V, \, E)$ with $E \neq \emptyset$ and ${\rm Aut} (\Gamma) \cong G$, which is finite if $G$ is finite; then ${\rm Aut} (X_\Gamma, \, R_\Gamma) \cong G$, $R_\Gamma$ is an FS-solution by \leref{RGammaFS}, and $X_\Gamma$ is finite if $\Gamma$ is finite.
\end{proof}

\begin{remark}\relabel{autFS}
The classes of groups realized by other restricted classes of FS-solutions can be much smaller. For instance, any bijective FS-solution is isomorphic to a solution on $A \times B$ of the form
\[
R_\omega \bigl( (a_1, \, b_1), \, (a_2, \, b_2) \bigr) = \bigl( (a_1, \, b_2), \, (a_2, \, \omega(b_1)) \bigr),
\]
with $\omega$ a permutation of $B$, and its automorphism group is ${\rm Sym} (A) \times C(\omega)$, where $C(\omega)$ is the centralizer of $\omega$ in ${\rm Sym} (B)$ \cite[Theorem 2.20]{acm1}.
\end{remark}

\thref{autFS} provides the theoretical answer we were looking for, but it is far from optimal with respect to the size of the underlying set. We show below that the dihedral group $D_4$ can be realized as the automorphism group of an FS-solution on a set with $4$ elements, while the smallest set on which such a solution can be obtained by applying the construction of \thref{autFS} has $12$ elements.

\begin{example}\exlabel{D4FS}
Let $X := \{1, \, 2, \, 3, \, 4\}$ and $f := (1 \,\, 2)(3 \,\, 4) \in {\rm Sym} (X)$. By \cite[Example 2.3(3)]{acm1}, the map $R_f \colon X \times X \to X \times X$, $R_f (x, \, y) := \bigl( y, \, f(x) \bigr)$, is a bijective FS-solution. By \cite[Theorem 2.19(3)]{acm1}, ${\rm Aut} (X, \, R_f)$ is the centralizer of $f$ in ${\rm Sym} (X)$. A permutation $\sigma$ commutes with $f$ if and only if it maps each of the two cycles $\{1, \, 2\}$ and $\{3, \, 4\}$ of $f$ onto a cycle of $f$; hence
\begin{equation*}
{\rm Aut} (X, \, R_f) = \bigl\{ {\rm Id}, \, (1 \,\, 2), \, (3 \,\, 4), \, (1 \,\, 2)(3 \,\, 4), \, (1 \,\, 3)(2 \,\, 4), \, (1 \,\, 4)(2 \,\, 3), \, (1 \,\, 3 \,\, 2 \,\, 4), \, (1 \,\, 4 \,\, 2 \,\, 3) \bigr\}.
\end{equation*}
This group is generated by $\rho := (1 \,\, 3 \,\, 2 \,\, 4)$, of order $4$, and $s := (1 \,\, 2)$, of order $2$, with $s \rho s = \rho^{-1}$; thus ${\rm Aut} (X, \, R_f)$ is isomorphic to the dihedral group $D_4$ of order $8$.
\end{example}

On the other hand, \thref{autFS} gives:

\begin{example}\exlabel{D4graph}
Let $\Gamma = (V, \, E)$ be the graph with four vertices $V = \{1, \, 2, \, 3, \, 4\}$ and two disjoint edges $e = \{1, \, 2\}$ and $f = \{3, \, 4\}$. An automorphism of $\Gamma$ either fixes both edges or interchanges them, and it may interchange the two endpoints of each edge; hence
\begin{equation*}
{\rm Aut} (\Gamma) = \bigl\{ {\rm Id}, \, (1 \,\, 2), \, (3 \,\, 4), \, (1 \,\, 2)(3 \,\, 4), \, (1 \,\, 3)(2 \,\, 4), \, (1 \,\, 4)(2 \,\, 3), \, (1 \,\, 3 \,\, 2 \,\, 4), \, (1 \,\, 4 \,\, 2 \,\, 3) \bigr\} \cong D_4,
\end{equation*}
the dihedral group of order $8$. The set $X_\Gamma$ has $4 \cdot 2 + 4 = 12$ elements: the eight pairs $(v, \, e)$ and $(v, \, f)$, with $v\in V$, and the four elements $x_{1, e}$, $x_{2, e}$, $x_{3, f}$, $x_{4, f}$. By \equref{RGamma}, the FS-solution $R_\Gamma$ is given by $R_\Gamma (x, \, y) = \bigl( (a(x), \, b(y)), \, (a(y), \, b(x)) \bigr)$; for instance,
\begin{equation*}
R_\Gamma \bigl( x_{1, e}, \, (3, \, f) \bigr) = \bigl( (1, \, f), \, (3, \, e) \bigr), \qquad R_\Gamma \bigl( x_{2, e}, \, x_{4, f} \bigr) = \bigl( (2, \, f), \, (4, \, e) \bigr).
\end{equation*}
By \thref{autFS}, ${\rm Aut} (X_\Gamma, \, R_\Gamma) \cong  {\rm Aut} (\Gamma) \cong D_4$. The set $X_\Gamma$ above is the smallest one which can be obtained from \thref{autFS} for the group $D_4$. Indeed, $|X_\Gamma| = |E| \, (|V| + 2)$, and if ${\rm Aut} (\Gamma) \cong D_4$, then $|V| \geq 4$, since ${\rm Aut} (\Gamma)$ is a subgroup of ${\rm Sym} (V)$, and $|E| \geq 2$, since a graph with exactly one edge $e$ has automorphism group $C_2 \times {\rm Sym} (V \setminus e)$, whose order $2 \, (|V| - 2)!$ is never equal to $8$.
\end{example}


\subsection{BLS-solutions}\sselabel{BLS}

In this subsection we construct, for any graph, a BLS-solution with the same automorphism group; hence the category of all BLS-solutions is also universal and finitely universal. First, we recall that
a map $R \colon X\times X \to X\times X$ is a solution of the \emph{Baaj--Long--Skandalis equation}, or a \emph{BLS-solution} \cite{ACM2, CMv1}, if
\[
R^{12}R^{23} = R^{23}R^{12}, \quad R^{12}R^{13} = R^{13}R^{12} \quad {\rm and} \quad R^{13}R^{23} = R^{23}R^{13}.
\]
Any BLS-solution is a YB-solution \cite[Remarks 2.2(3)]{ACM2}, and the category of BLS-solutions is isomorphic to the variety of P{\l}onka bi-magmas \cite[Theorem 2.4]{ACM2}. The structure of P{\l}onka bi-magmas is described in \cite[Theorem 3.25]{ACMv3}: a P{\l}onka bi-magma structure on a set $X$ is given by a partition $X = \coprod_{i\in I} X_i$ and two families of maps $f_i^j$, $g_i^j \colon X_i \to X_i$ ($i$, $j\in I$) which commute pairwise for each fixed $i$, via $x_i \cdot x_j := f_i^j (x_i)$ and $x_i \ast x_j := g_j^i (x_j)$, for $x_i\in X_i$, $x_j\in X_j$.

Let $\Gamma = (V, \, E)$ be a graph. For $v\in V$ we denote by $N[v] := \{v\} \cup \{u \in V \, | \, u \sim v\}$ the \emph{closed neighbourhood} of $v$, i.e. the set consisting of $v$ and of all the vertices adjacent to $v$. We define the set

\begin{equation*}
Y_\Gamma := \{p_v \, | \, v \in V\} \, \sqcup \, \{q_v \, | \, v \in V\},
\end{equation*}
containing two elements $p_v \neq q_v$ for each vertex $v$, so that $|Y_\Gamma| = 2 |V|$, and we denote by $\nu \colon Y_\Gamma \to V$ the map with $\nu (p_v) = \nu (q_v) = v$. We define two
multiplications on $Y_\Gamma$ by
\begin{equation}\eqlabel{BLSGamma}
x \cdot y := \begin{cases} p_{\nu(x)}, & \text{if} \,\,\, \nu(y) \in N[\nu(x)],\\ x, & \text{otherwise}, \end{cases} \qquad \qquad x \ast y := y,
\end{equation}
for all $x$, $y\in Y_\Gamma$. We denote by
\begin{equation}\eqlabel{blsgraf}
R'_\Gamma \colon Y_\Gamma \times Y_\Gamma \to Y_\Gamma \times Y_\Gamma, \qquad R'_\Gamma (x, \, y) = (x \cdot y, \, y)
\end{equation}
for all $x$, $y\in Y_\Gamma$, the associated map of the multiplications \equref{BLSGamma}.

\begin{lemma}\lelabel{BLSGamma}
Let $\Gamma = (V, \, E)$ be a graph. Then the map $R'_\Gamma \colon Y_\Gamma \times Y_\Gamma \to Y_\Gamma \times Y_\Gamma$ given by \equref{blsgraf} is an idempotent BLS-solution on the set $Y_\Gamma$; in particular, it is a YB-solution.
\end{lemma}

\begin{proof}
$R'_\Gamma$ is a BLS-solution by the structure theorem of BLS-solutions \cite[Theorem 2.6]{ACM2}, applied to the partition $Y_\Gamma = \coprod_{v\in V} \{p_v, \, q_v\}$, with $f_v^u$ the constant map with value $p_v$ if $u \in N[v]$, $f_v^u := {\rm Id}$ otherwise, and $g_v^u := {\rm Id}$ (all these maps commute). For $u\in V$, let $r_u \colon Y_\Gamma \to Y_\Gamma$ be the map given by $r_u (x) := x \cdot y$, for any $y$ with $\nu(y) = u$. By \equref{BLSGamma}, $r_u$ does not depend on the choice of $y$, it maps each set $\{p_v, \, q_v\}$ into itself, and its restriction to $\{p_v, \, q_v\}$ is either the identity or the constant map with value $p_v$; in particular, $r_u$ is idempotent. Hence, for any $x$, $y\in Y_\Gamma$ we have
\begin{equation*}
(R'_\Gamma)^2 (x, \, y) = \bigl( (x \cdot y) \cdot y, \, y \bigr) = \bigl( r_{\nu(y)}^2 (x), \, y \bigr) = \bigl( r_{\nu(y)} (x), \, y \bigr) = (x \cdot y, \, y) = R'_\Gamma (x, \, y). \qedhere
\end{equation*}
\end{proof}

With these preparations, we can now prove the following:

\begin{theorem}\thlabel{autBLS}
For any graph $\Gamma = (V, \, E)$, the map
\begin{equation*}
{\rm Aut} (\Gamma) \to {\rm Aut} (Y_\Gamma, \, R'_\Gamma), \qquad \pi \mapsto \widetilde{\pi}, \qquad \widetilde{\pi} (p_v) := p_{\pi(v)}, \quad \widetilde{\pi} (q_v) := q_{\pi(v)},
\end{equation*}
is an isomorphism of groups. In particular, any group is isomorphic to the automorphism group of an idempotent BLS-solution, and hence of a YB-solution; if the group is finite, the solution can be chosen on a finite set: if $G \cong {\rm Aut} (\Gamma)$ for a finite graph $\Gamma$ with $n$ vertices, then $|Y_\Gamma| = 2n$.
\end{theorem}

\begin{proof}
Let $\pi \in {\rm Aut} (\Gamma)$ be an automorphism of the graph $\Gamma = (V, \, E)$. Since $\pi$ preserves adjacency, $u \in N[v]$ if and only if $\pi(u) \in N[\pi(v)]$. Moreover, $\nu \circ \widetilde{\pi} = \pi \circ \nu$, and $\widetilde{\pi}$ maps the elements $p_v$ (respectively $q_v$) to elements of the same kind. Using \equref{BLSGamma}, for $x$, $y\in Y_\Gamma$ we obtain: if $\nu(y) \in N[\nu(x)]$, then $\nu (\widetilde{\pi} (y)) \in N[\nu (\widetilde{\pi}(x))]$ and $\widetilde{\pi} (x \cdot y) = p_{\pi (\nu(x))} = \widetilde{\pi} (x) \cdot \widetilde{\pi} (y)$; otherwise $\widetilde{\pi} (x \cdot y) = \widetilde{\pi}(x) = \widetilde{\pi} (x) \cdot \widetilde{\pi} (y)$. Obviously $\widetilde{\pi} (x \ast y) = \widetilde{\pi} (x) \ast \widetilde{\pi} (y)$. Thus $\widetilde{\pi}$ is an automorphism of the BLS-solution $(Y_\Gamma, \, R'_\Gamma)$ (its inverse is $\widetilde{\pi^{-1}}$), and $\pi \mapsto \widetilde{\pi}$ is an injective morphism of groups.

It remains to prove the surjectivity. Let $\sigma$ be an automorphism of the bi-magma $(Y_\Gamma, \, \cdot, \, \ast)$ associated to $R'_\Gamma$. We use three intrinsic properties of the multiplications given by \equref{BLSGamma}:
\begin{enumerate}[(a)]
\item The elements $p_v$ are exactly the \emph{left zeros} of $(Y_\Gamma, \, \cdot)$, i.e. the elements $z$ such that $z \cdot y = z$ for all $y$. Indeed, $p_v \cdot y = p_v$ for all $y$, while $q_v \cdot q_v = p_v \neq q_v$ (since $v \in N[v]$).
\item $q_v \cdot q_v = p_v$, for all $v\in V$.
\item For $u \neq v$, we have $u \sim v$ if and only if $q_v \cdot p_u \neq q_v$.
\end{enumerate}
By (a), $\sigma$ maps the set $\{p_v \, | \, v \in V\}$ onto itself, hence also the set $\{q_v \, | \, v\in V\}$ onto itself. Thus there is a bijection $\pi \colon V \to V$ with $\sigma (q_v) = q_{\pi(v)}$ for all $v$, and by (b), $\sigma (p_v) = \sigma (q_v \cdot q_v) = q_{\pi(v)} \cdot q_{\pi(v)} = p_{\pi(v)}$, i.e. $\sigma = \widetilde{\pi}$. Finally, for $u \neq v$, using (c) twice and the fact that $\sigma$ is an injective morphism,
\begin{equation*}
u \sim v \,\, \Longleftrightarrow \,\, q_v \cdot p_u \neq q_v \,\, \Longleftrightarrow \,\, \sigma(q_v) \cdot \sigma (p_u) \neq \sigma (q_v) \,\, \Longleftrightarrow \,\, q_{\pi(v)} \cdot p_{\pi(u)} \neq q_{\pi(v)} \,\, \Longleftrightarrow \,\, \pi(u) \sim \pi(v).
\end{equation*}
Hence $\pi \in {\rm Aut} (\Gamma)$ and $\sigma = \widetilde{\pi}$. The last statement follows from \thref{Frucht} and \leref{BLSGamma}.
\end{proof}

\begin{remark}\relabel{autBLS}
Comparing \thref{autFS} and \thref{autBLS}: the YB-solution $R'_\Gamma$ lives on a set with $2 |V|$ elements, while the FS-solution $R_\Gamma$ lives on a set with $|V| \, |E| + 2 |E|$ elements; on the other hand, $R_\Gamma$ belongs to the class of FS-solutions, which is much more rigid than the class of BLS-solutions. For instance, on a set with $3$ elements there are $14$ FS-solutions and $55$ BLS-solutions, out of a total of $1045$ YB-solutions, up to isomorphism \cite[Table 2]{MilBIP}.
\end{remark}

\subsection{Bipartite solutions}\sselabel{bij}

The solutions constructed in \sseref{FS} and \sseref{BLS} are not bijective: if $E \neq \emptyset$, the image of $R_\Gamma$ is contained in $(V\times E) \times (V \times E) \subsetneq X_\Gamma \times X_\Gamma$, and if $V \neq \emptyset$, then $R'_\Gamma$ is idempotent but not the identity, since $R'_\Gamma (q_v, \, q_v) = (p_v, \, q_v)$. Moreover, as recalled in \reref{autFS}(2), the automorphism groups of bijective FS-solutions are of the very special form ${\rm Sym} (A) \times C(\omega)$. In this subsection we show that, nevertheless, any group is the automorphism group of a bijective solution of the braid equation, which is even involutive, i.e. $R^2 = {\rm Id}$; equivalently, of a unitary YB-solution. The solutions we use belong to the class of \emph{bipartite solutions}, introduced and studied in \cite{MilBIP}.

A \emph{bipartite datum} \cite[Definition 2.1]{MilBIP} is a triple $(Y, \, Z, \, g)$, where $Y$ and $Z$ are two disjoint sets and $g \colon Y \to {\rm Sym} (Z)$, $y \mapsto g_y$, is an arbitrary map, i.e. a family $(g_y)_{y\in Y}$ of permutations of $Z$; no relation between the permutations $g_y$ is assumed. The datum is called \emph{split} if $g_y = {\rm Id}_Z$ for all $y\in Y$, and \emph{non-split} otherwise. The \emph{bipartite solution} associated to $(Y, \, Z, \, g)$ is the map $R_g \colon X \times X \to X \times X$, where $X := Y \sqcup Z$, given by
\begin{equation}\eqlabel{Rg}
R_g (x, \, x') := \begin{cases} (x, \, x'), & \text{if} \,\,\, x, \, x' \in Y \,\,\, \text{or} \,\,\, x, \, x'\in Z,\\
\bigl( g_y (z), \, y \bigr), & \text{if} \,\,\, x = y \in Y, \,\, x' = z \in Z,\\
\bigl( y, \, g_y^{-1} (z) \bigr), & \text{if} \,\,\, x = z \in Z, \,\, x' = y \in Y.
\end{cases}
\end{equation}
In terms of the bi-magma $(X, \, \cdot, \, \ast)$ associated to $R_g$ by \equref{notatieR}, for all $y$, $y'\in Y$ and $z$, $z'\in Z$ we have:
\begin{equation}\eqlabel{bimagRg}
\begin{split}
& y \cdot y' = y, \quad y \ast y' = y', \quad z \cdot z' = z, \quad z \ast z' = z',\\
& y \cdot z = g_y (z), \quad y \ast z = y, \quad z \cdot y = y, \quad z \ast y = g_y^{-1} (z).
\end{split}
\end{equation}

We recall below the two results of \cite[Theorem 2.2 and Proposition 2.15]{MilBIP} that we need; applied to the transpositions given by the edges of a graph, they yield the main result of this subsection: any group is isomorphic to the automorphism group of a bijective solution $R$ of the braid equation with $R^2 = {\rm Id}$ (and of a bijective and unitary YB-solution).

\begin{proposition}\prlabel{Rg}
For any bipartite datum $(Y, \, Z, \, g)$, the map $R_g$ is a bijective solution of the braid equation and $R_g^2 = {\rm Id}_{X \times X}$. Consequently, $\tau_X \circ R_g$ is a bijective and unitary YB-solution.
\end{proposition}

We also need the description of the automorphism group of $(X, \, R_g)$:

\begin{proposition}\prlabel{autRg}
Let $(Y, \, Z, \, g)$ be a non-split bipartite datum and $X = Y \sqcup Z$. Then the map
\begin{equation*}
\begin{split}
& H := \Bigl\{ (\alpha, \, \beta) \in {\rm Sym} (Y) \times {\rm Sym} (Z) \, \bigl| \, \beta \circ g_y \circ \beta^{-1} = g_{\alpha(y)}, \,\, \text{for all} \,\, y\in Y \Bigr\} \to {\rm Aut} (X, \, R_g),\\
& (\alpha, \, \beta) \mapsto \alpha \sqcup \beta,
\end{split}
\end{equation*}
is an isomorphism of groups, where $\alpha \sqcup \beta$ is the bijection of $X = Y \sqcup Z$ which is equal to $\alpha$ on $Y$ and to $\beta$ on $Z$; the left hand side is a subgroup of ${\rm Sym} (Y) \times {\rm Sym} (Z)$.
\end{proposition}

The structure of the group $H$ is described in \cite[Theorem 2.16(2)]{MilBIP}. In particular, if the map $y \mapsto g_y$ is injective, then in each pair $(\alpha, \, \beta) \in H$ the permutation $\alpha$ is determined by $\beta$, and ${\rm Aut} (X, \, R_g)$ is isomorphic to the normalizer in ${\rm Sym} (Z)$ of the set $\{g_y \, | \, y \in Y\}$ \cite[Theorem 2.16(3)]{MilBIP}. We apply this to the set of transpositions given by the edges of a graph.

Let $\Gamma = (V, \, E)$ be a graph with $E \neq \emptyset$. For an edge $e = \{u, \, v\} \in E$ let $g_e := (u \,\, v) \in {\rm Sym} (V)$ be the transposition of its endpoints; thus, for $w\in V$,
\begin{equation*}
g_e (w) = \begin{cases} \text{the other endpoint of} \,\, e, & \text{if} \,\,\, w \in e,\\ w, & \text{if} \,\,\, w \notin e. \end{cases}
\end{equation*}
We apply the above construction to the bipartite datum $(E, \, V, \, g)$, i.e. with $Y := E$, $Z := V$ and $g := (g_e)_{e\in E}$, and we denote $X''_\Gamma := E \sqcup V$ and $R''_\Gamma := R_g$. Explicitly, for $e$, $e'\in E$ and $v$, $v'\in V$ the involutive solution of the braid equation is given by:
\begin{equation}\eqlabel{Rsecond}
R''_\Gamma (e, \, e') = (e, \, e'), \quad R''_\Gamma (v, \, v') = (v, \, v'), \quad R''_\Gamma (e, \, v) = \bigl( g_e (v), \, e \bigr), \quad R''_\Gamma (v, \, e) = \bigl( e, \, g_e (v) \bigr).
\end{equation}

\begin{theorem}\thlabel{autbij}
Let $\Gamma = (V, \, E)$ be a graph with $E \neq \emptyset$. Then $R''_\Gamma$ is a bijective solution of the braid equation on $X''_\Gamma = E \sqcup V$ with $(R''_\Gamma)^2 = {\rm Id}$, and the map
\begin{equation*}
{\rm Aut} (\Gamma) \to {\rm Aut} (X''_\Gamma, \, R''_\Gamma), \qquad \pi \mapsto \widehat{\pi}, \qquad \widehat{\pi} (v) := \pi(v), \quad \widehat{\pi} (\{u, \, v\}) := \{\pi(u), \, \pi(v)\},
\end{equation*}
is an isomorphism of groups.
\end{theorem}

\begin{proof}
The first statement follows from \prref{Rg}. Let $\pi \in {\rm Aut} (\Gamma)$ and let $\pi_E \colon E \to E$, $\pi_E (\{u, \, v\}) := \{\pi(u), \, \pi(v)\}$, be the induced bijection of $E$ (its inverse is $(\pi^{-1})_E$). For any edge $e = \{u, \, v\}$ we have
\begin{equation}\eqlabel{conjtransp}
\pi \circ (u \,\, v) \circ \pi^{-1} = \bigl( \pi(u) \,\, \pi(v) \bigr), \quad \text{i.e.} \quad \pi \, g_e \, \pi^{-1} = g_{\pi_E (e)}.
\end{equation}
Indeed, both sides map $\pi(u) \mapsto \pi(v) \mapsto \pi(u)$ and fix every other vertex. Thus $(\pi_E, \, \pi)$ belongs to the group $H$ of \prref{autRg}, and $\widehat{\pi} = \pi_E \sqcup \pi$ is an automorphism of $(X''_\Gamma, \, R''_\Gamma)$. Clearly $\widehat{\pi \pi'} = \widehat{\pi} \, \widehat{\pi'}$, and $\pi \mapsto \widehat{\pi}$ is injective, since $\widehat{\pi}$ restricted to $V$ is $\pi$.

For the surjectivity, note that the bipartite datum $(E, \, V, \, g)$ is non-split: if $e_0 = \{u_0, \, v_0\} \in E$, then $g_{e_0} (u_0) = v_0 \neq u_0$. Let $\varphi \in {\rm Aut} (X''_\Gamma, \, R''_\Gamma)$. By \prref{autRg}, $\varphi = \alpha \sqcup \beta$ with $\alpha \in {\rm Sym} (E)$, $\beta\in {\rm Sym} (V)$ and $\beta \, g_e \, \beta^{-1} = g_{\alpha(e)}$ for all $e\in E$. For $e = \{u, \, v\}$, the left hand side is the transposition $(\beta(u) \,\, \beta(v))$, by the same computation as in \equref{conjtransp}, while the right hand side is the transposition of the endpoints of $\alpha(e)$. Since a transposition determines the two points it interchanges, we obtain
\begin{equation}\eqlabel{alphabeta}
\alpha (\{u, \, v\}) = \{\beta(u), \, \beta(v)\}, \qquad \text{for all} \,\, \{u, \, v\} \in E.
\end{equation}
We show that $\beta \in {\rm Aut} (\Gamma)$. If $\{u, \, v\} \in E$, then $\{\beta(u), \, \beta(v)\} = \alpha (\{u, \, v\}) \in E$ by \equref{alphabeta}. Conversely, let $u$, $v\in V$ with $\{\beta(u), \, \beta(v)\} \in E$. Since $\alpha$ is surjective, $\{\beta(u), \, \beta(v)\} = \alpha (e')$ for some $e' = \{u', \, v'\} \in E$, and by \equref{alphabeta}, $\{\beta(u), \, \beta(v)\} = \{\beta(u'), \, \beta(v')\}$; as $\beta$ is injective, $\{u, \, v\} = \{u', \, v'\} \in E$. Thus $\pi := \beta \in {\rm Aut} (\Gamma)$, \equref{alphabeta} says that $\alpha = \pi_E$, and therefore $\varphi = \pi_E \sqcup \pi = \widehat{\pi}$.
\end{proof}

\begin{corollary}\colabel{bij}
Any group $G$ is isomorphic to the automorphism group of a bijective solution $R$ of the braid equation with $R^2 = {\rm Id}$, and to the automorphism group of a bijective and unitary YB-solution. If $G$ is finite, the solutions can be chosen on finite sets: if $G \cong {\rm Aut} (\Gamma)$ for a finite graph $\Gamma$ with $n$ vertices and $m \geq 1$ edges, the solutions live on a set with $n + m$ elements.
\end{corollary}

\begin{proof}
By \thref{Frucht} and the remark following it, $G \cong {\rm Aut} (\Gamma)$ for some graph $\Gamma = (V, \, E)$ with $E \neq \emptyset$, which is finite if $G$ is finite; if $G$ is trivial, we can take for $\Gamma$ the graph of \exref{autbij} below. Now apply \thref{autbij} to $R''_\Gamma$, and \equref{auttau} to $\tau_{X''_\Gamma} \circ R''_\Gamma$, which is a bijective and unitary YB-solution by \prref{Rg}.
\end{proof}

\begin{example}\exlabel{autbij}
Let $\Gamma$ be the graph with vertices $0, \ldots, 5$ and the six edges $\{0, \, 1\}$, $\{1, \, 2\}$, $\{0, \, 2\}$, $\{2, \, 3\}$, $\{3, \, 4\}$, $\{1, \, 5\}$, i.e. a triangle with a path of length $2$ attached at the vertex $2$ and a path of length $1$ attached at the vertex $1$. Its vertices $0, \ldots, 5$ have degrees $2$, $3$, $3$, $2$, $1$, $1$, respectively. The two vertices of degree $1$ are $4$, whose neighbour $3$ has degree $2$, and $5$, whose neighbour $1$ has degree $3$; hence any automorphism $\pi$ fixes $4$ and $5$, and therefore also their neighbours $3$ and $1$. Then $\pi$ fixes $2$, the other vertex of degree $3$, and finally $0$. Thus ${\rm Aut}(\Gamma)$ is trivial. By \thref{autbij}, $(X''_\Gamma, \, R''_\Gamma)$ is a bijective solution of the braid equation on $12$ elements with $(R''_\Gamma)^2 = {\rm Id}$ and with trivial automorphism group.
\end{example}

\begin{remark}\relabel{autbij}
The construction $\Gamma \mapsto (X''_\Gamma, \, R''_\Gamma)$ reflects isomorphisms: if $\Gamma = (V, \, E)$ and $\Gamma' = (V', \, E')$ are two graphs with at least one edge, then $(X''_\Gamma, \, R''_\Gamma) \cong (X''_{\Gamma'}, \, R''_{\Gamma'})$ if and only if $\Gamma \cong \Gamma'$. Indeed, both bipartite data are non-split, so, by \cite[Theorem 2.12(3)]{MilBIP}, the two solutions are isomorphic if and only if there exist bijections $\alpha \colon E \to E'$ and $\beta \colon V \to V'$ such that $\beta \, g_e \, \beta^{-1} = g'_{\alpha(e)}$, for all $e\in E$. Exactly as in the proof of \thref{autbij}, this means that $\alpha (\{u, \, v\}) = \{\beta(u), \, \beta(v)\}$, for all $\{u, \, v\} \in E$, i.e. that $\beta$ is an isomorphism of graphs $\Gamma \to \Gamma'$ and $\alpha$ is the induced bijection between the sets of edges.
\end{remark}

Collecting the results of this section, we obtain the theorem stated in the Introduction.

\begin{theorem}\thlabel{mainth}
Any group $G$ is isomorphic to the automorphism group of a set-theoretic solution of the Yang--Baxter equation, and to the automorphism group of a set-theoretic solution of the braid equation. If $G$ is finite, then the solutions can be chosen on finite sets. Moreover:
\begin{enumerate}[(1)]
\item the solution of the braid equation can be chosen to be a solution of the Frobenius-Separability equation;
\item the solution of the Yang--Baxter equation can be chosen to be an idempotent solution of the Baaj--Long--Skandalis equation;
\item the solution of the braid equation can be chosen to be a bipartite solution, in particular an involutive one, and the solution of the Yang--Baxter equation can be chosen to be unitary.
\end{enumerate}
\end{theorem}

\begin{proof}
Each of the statements (1)-(3) contains the first two statements of the theorem. (1) follows from \thref{autFS}, (2) from \leref{BLSGamma} and \thref{autBLS}, and (3) from \coref{bij}; in all cases the solutions can be chosen on finite sets if $G$ is finite.
\end{proof}

We close this subsection with the following result, which may be of independent interest: a construction of unitary YB-solutions arising from any field $K$. 

\begin{theorem}\thlabel{field}
For any field $K$ there exists a unitary YB-solution $(X, \, r)$ such that ${\rm Aut} (X, \, r) \cong {\rm Aut} (K)$.
\end{theorem}

\begin{proof}
If $|K| = 2$, then ${\rm Aut} (K)$ is trivial and we can take the one-element solution. Assume that $|K| \geq 3$. Let $Z := K$, $Y := \{y_1\} \sqcup \{y_u \, | \, u \in K \setminus \{0, \, 1\}\}$, and consider the family of maps $g_y \colon K \to K$, $y \in Y$, defined by
\begin{equation*}
g_{y_1} (z) := z + 1, \qquad g_{y_u} (z) := uz,
\end{equation*}
for all $z \in K$ and $u \in K \setminus \{0, \, 1\}$; we denote by $t_1 := g_{y_1}$ the translation by $1$ and by $m_u := g_{y_u}$ the multiplication by $u$. Then $X := Y \sqcup K$ has $2 |K| - 1$ elements if $K$ is finite, and the unitary YB-solution $r \colon X \times X \to X \times X$ given by \equref{rg} takes the following form:
\begin{alignat*}{2}
r (y_1, \, z) &= (y_1, \, z + 1), &\qquad r (z, \, y_1) &= (z - 1, \, y_1),\\
r (y_u, \, z) &= (y_u, \, uz), & r (z, \, y_u) &= (u^{-1} z, \, y_u),
\end{alignat*}
for all $z \in K$ and $u \in K \setminus \{0, \, 1\}$, and $r$ is the flip map on $(Y \times Y) \cup (K \times K)$. We show that
\begin{equation*}
{\rm Aut} (X, \, r) = \{ \widetilde{\sigma} \, | \, \sigma \in {\rm Aut} (K) \} \cong {\rm Aut} (K),
\end{equation*}
where $\widetilde{\sigma}$ is equal to $\sigma$ on $K$, $\widetilde{\sigma} (y_1) := y_1$ and $\widetilde{\sigma} (y_u) := y_{\sigma(u)}$.

The datum is non-split and the map $y \mapsto g_y$ is injective. By \equref{autnorm}, it is enough to show that $N (G_g) = {\rm Aut} (K)$. If $\sigma \in {\rm Aut} (K)$, then $\sigma \, t_1 \, \sigma^{-1} = t_1$ and $\sigma \, m_u \, \sigma^{-1} = m_{\sigma(u)}$; hence $\sigma \in N (G_g)$, and the corresponding automorphism of $(X, \, r)$ is $\widetilde{\sigma}$.

Conversely, let $\beta \in N (G_g)$. Since conjugation preserves the number of fixed points and $t_1$ is the only element of $G_g$ without fixed points, we have $\beta \, t_1 \, \beta^{-1} = t_1$, i.e. $\beta (z + 1) = \beta (z) + 1$ for all $z \in K$. Moreover, let $u \in K \setminus \{0, \, 1\}$; such elements exist, as $|K| \geq 3$. The map $\beta \, m_u \, \beta^{-1} \in G_g$ has exactly one fixed point, namely $\beta (0)$; hence it is not $t_1$, i.e. $\beta \, m_u \, \beta^{-1} = m_{u'}$ for some $u' \in K \setminus \{0, \, 1\}$. Since the only fixed point of $m_{u'}$ is $0$, we obtain $\beta (0) = 0$, and therefore $\beta (1) = \beta (0) + 1 = 1$. Furthermore, $\beta \, m_u = m_{u'} \, \beta$, i.e. $\beta (uz) = u' \beta (z)$ for all $z \in K$; taking $z = 1$ we obtain $u' = \beta (u)$.

Hence $\beta (uz) = \beta (u) \beta (z)$ for all $u$, $z \in K$, the cases $u \in \{0, \, 1\}$ being obvious. Finally, for $z \in K$ and $w \in K \setminus \{0\}$ we have $z + w = w (w^{-1} z + 1)$, and therefore
\begin{equation*}
\beta (z + w) = \beta (w) \bigl( \beta (w^{-1}) \beta (z) + 1 \bigr) = \beta (z) + \beta (w),
\end{equation*}
since $\beta (w) \beta (w^{-1}) = \beta (1) = 1$; for $w = 0$ this equality is obvious. Thus $\beta$ is a bijective map which is additive and multiplicative, i.e. $\beta \in {\rm Aut} (K)$.
\end{proof}

\begin{remark}\relabel{field2}
Combined with a theorem of Fried and Koll\'ar, which states that every group is isomorphic to the automorphism group of a field \cite[Theorem 4.2]{Babai1995}, \thref{field} gives a second proof, which avoids graphs, of the fact that any group is isomorphic to the automorphism group of a bipartite solution (\coref{bij}).
\end{remark}

\subsection{Examples}\sselabel{examples}

We end the section with explicit solutions of the Yang--Baxter equation whose automorphism groups are the cyclic group $C_n$, the symmetric group $S_n$, the dihedral group $D_n$ of order $2n$ and the quaternion group $Q_8$, and with a unitary YB-solution, constructed directly from a field $K$, whose automorphism group is ${\rm Aut} (K)$. All of them are bipartite solutions, and we describe them directly as YB-solutions. Let $(Y, \, Z, \, g)$ be a bipartite datum and $X = Y \sqcup Z$. The YB-solution $r_g := \tau_X \circ R_g$, where $R_g$ is given by \equref{Rg}, is
\begin{equation}\eqlabel{rg}
r_g (x, \, x') := \begin{cases} (x', \, x), & \text{if} \,\,\, x, \, x' \in Y \,\,\, \text{or} \,\,\, x, \, x'\in Z,\\
\bigl( y, \, g_y (z) \bigr), & \text{if} \,\,\, x = y \in Y, \,\, x' = z \in Z,\\
\bigl( g_y^{-1} (z), \, y \bigr), & \text{if} \,\,\, x = z \in Z, \,\, x' = y \in Y.
\end{cases}
\end{equation}
By \prref{Rg}, $r_g$ is a unitary YB-solution, and by \equref{auttau}, ${\rm Aut} (X, \, r_g) = {\rm Aut} (X, \, R_g)$. In all the examples below the datum $(Y, \, Z, \, g)$ is non-split and the map $y \mapsto g_y$ is injective; hence, by \cite[Theorem 2.16(3)]{MilBIP},
\begin{equation}\eqlabel{autnorm}
{\rm Aut} (X, \, r_g) \cong N (G_g) := \bigl\{ \beta \in {\rm Sym} (Z) \, \bigl| \, \beta \, G_g \, \beta^{-1} = G_g \bigr\}, \qquad G_g := \{ g_y \, | \, y \in Y \}.
\end{equation}
More precisely, by \prref{autRg}, any $\beta \in N (G_g)$ gives the automorphism $\alpha_\beta \sqcup \beta$ of $(X, \, r_g)$, where $\alpha_\beta (y)$ is the unique element $y' \in Y$ such that $g_{y'} = \beta \, g_y \, \beta^{-1}$, and all automorphisms are of this form. Consequently, such a solution is completely described by the list of the permutations $g_y$.

\begin{example}[The cyclic group $C_n$]\exlabel{Cn}
Let $n \geq 2$, $Z := \bZ_n = \{0, \, 1, \ldots, n-1\}$, $Y := \{y\}$ and $g_y (z) := z + 1$ (all the computations in $Z$ are modulo $n$). Then $X = \{y, \, 0, \, 1, \ldots, n-1\}$ has $n + 1$ elements, and the YB-solution \equref{rg} is
\begin{equation*}
r (y, \, z) = (y, \, z + 1), \qquad r (z, \, y) = (z - 1, \, y), \qquad r (y, \, y) = (y, \, y), \qquad r (z, \, z') = (z', \, z),
\end{equation*}
for all $z$, $z' \in \bZ_n$. Then ${\rm Aut} (X, \, r)$ is the cyclic group of order $n$ generated by the bijection $\sigma$ of $X$ given by $\sigma (y) = y$ and $\sigma (z) = z + 1$, for all $z \in \bZ_n$. Indeed, the datum is non-split and $Y$ has only one element, so by \equref{autnorm} the automorphisms of $(X, \, r)$ are exactly the bijections ${\rm Id}_Y \sqcup \beta$, where $\beta$ is a permutation of $\bZ_n$ commuting with the cycle $c := g_y$, i.e. $\beta (z + 1) = \beta (z) + 1$ for all $z \in \bZ_n$. By induction on $z$, any such $\beta$ satisfies $\beta (z) = z + k$, where $k := \beta(0)$; conversely, each translation $z \mapsto z + k$ commutes with $c$. Hence ${\rm Aut} (X, \, r) = \{\sigma^k \, | \, k \in \bZ_n\}$, and $k \mapsto \sigma^k$ is an isomorphism of groups $\bZ_n \cong {\rm Aut} (X, \, r)$; that is, ${\rm Aut} (X, \, r) \cong C_n$.
\end{example}

\begin{example}[The symmetric group $S_n$]\exlabel{Sn}
Let $n \geq 2$, $Z := \{1, \, 2, \ldots, n\}$, $Y := \{ t_{ij} \, | \, 1 \leq i < j \leq n\}$ and $g_{t_{ij}} := (i \,\, j)$, the transposition of $i$ and $j$. Then $X = Y \sqcup Z$ has $n(n+1)/2$ elements, and the YB-solution \equref{rg} is
\begin{equation*}
r (t_{ij}, \, z) = \bigl( t_{ij}, \, (i \,\, j) (z) \bigr), \qquad r (z, \, t_{ij}) = \bigl( (i \,\, j) (z), \, t_{ij} \bigr),
\end{equation*}
for all $z\in Z$ and $i < j$, and $r$ is the flip map on $(Y \times Y) \cup (Z \times Z)$. The set $G_g$ is the set of all transpositions of $Z$, which is a conjugacy class of ${\rm Sym} (Z)$; hence $N (G_g) = {\rm Sym} (Z)$ and ${\rm Aut} (X, \, r) \cong S_n$. We note that $S_n$ is also, trivially, the automorphism group of the flip map $\tau_X$ on a set $X$ with $n$ elements (this is the case $Y = \emptyset$ in \equref{rg}).
\end{example}

\begin{example}[The dihedral group $D_n$]\exlabel{Dn}
Let $n \geq 3$, $Z := \bZ_n$, $Y := \{y_+, \, y_-\}$, $g_{y_+} (z) := z + 1$ and $g_{y_-} (z) := z - 1$. Then $X = \{y_+, \, y_-, \, 0, \, 1, \ldots, n-1\}$ has $n + 2$ elements, and the YB-solution \equref{rg} is
\begin{equation*}
r (y_{\pm}, \, z) = (y_{\pm}, \, z \pm 1), \qquad r (z, \, y_{\pm}) = (z \mp 1, \, y_{\pm}),
\end{equation*}
for all $z \in \bZ_n$, and $r$ is the flip map on $(Y \times Y) \cup (Z \times Z)$. Then ${\rm Aut} (X, \, r)$ is isomorphic to the dihedral group $D_n$ of order $2n$: its elements are the $n$ bijections $z \mapsto z + v$, which fix $y_+$ and $y_-$, and the $n$ bijections $z \mapsto -z + v$, which interchange $y_+$ and $y_-$ ($v \in \bZ_n$). Indeed, let $c := g_{y_+}$; then $g_{y_-} = c^{-1} \neq c$, since $n \geq 3$, and $\beta \in N (G_g)$ if and only if $\beta c \beta^{-1} \in \{c, \, c^{-1}\}$. If $\beta c \beta^{-1} = c$, then $\beta (z) = z + v$, as in \exref{Cn}. If $\beta c \beta^{-1} = c^{-1}$, let $s(z) := -z$; since $s c s^{-1} = c^{-1}$, the permutation $\beta s$ commutes with $c$, hence $\beta s (z) = z + v$, i.e. $\beta (z) = -z + v$, for some $v\in \bZ_n$. Conversely, all these permutations belong to $N(G_g)$, and the ones of the second type interchange $c$ and $c^{-1}$. 
\end{example}

\begin{example}[The quaternion group $Q_8$]\exlabel{Q8}
Let $Q_8 = \{\pm 1, \, \pm i, \, \pm j, \, \pm k\}$, with $i^2 = j^2 = k^2 = ijk = -1$, and let $A := \{\pm 1, \, \pm i\}$ be the cyclic subgroup generated by $i$ and $B := Q_8 \setminus A = \{\pm j, \, \pm k\}$. Let $Z := Q_8$, $Y := \{a, \, b, \, c\}$ and
\begin{equation*}
g_a (z) := zj, \qquad g_b (z) := \begin{cases} zi, & \text{if} \,\,\, z\in A,\\ z, & \text{if} \,\,\, z\in B,\end{cases} \qquad g_c (z) := \begin{cases} z, & \text{if} \,\,\, z\in A,\\ zi, & \text{if} \,\,\, z\in B,\end{cases}
\end{equation*}
i.e. $g_a$ is the right multiplication by $j$, while $g_b$ and $g_c$ are the two halves of the right multiplication by $i$. Then $X = \{a, \, b, \, c\} \sqcup Q_8$ has $11$ elements, $r$ is given by \equref{rg}, and ${\rm Aut} (X, \, r) = \{\sigma_q \, | \, q \in Q_8\} \cong Q_8$, where $\sigma_q (z) := qz$, for all $z\in Q_8$, $\sigma_q (a) := a$, and $\sigma_q$ fixes $b$ and $c$ if $q\in A$, respectively interchanges $b$ and $c$ if $q\in B$.

Indeed, for $t\in Q_8$ we denote by $\rho_t (z) := zt$ and $\lambda_t (z) := tz$ the right and the left multiplications by $t$; thus $g_a = \rho_j$ and $g_b \, g_c = g_c \, g_b = \rho_i$. We show that $N (G_g) = \{ \lambda_q \, | \, q\in Q_8 \}$. Any $\lambda_q$ commutes with all the right multiplications; moreover, $\lambda_q (A) = qA$ is equal to $A$ if $q\in A$, and to $B$ if $q\in B$. Hence $\lambda_q \, g_a \, \lambda_q^{-1} = g_a$, while $\lambda_q \, g_b \, \lambda_q^{-1}$, which is equal to $\rho_i$ on $qA$ and to the identity on $qB$, is $g_b$ if $q \in A$ and $g_c$ if $q\in B$; similarly for $g_c$. Thus $\lambda_q \in N(G_g)$, and the corresponding automorphism is $\sigma_q$. Conversely, let $\beta \in N(G_g)$. The permutation $g_a$ moves all the $8$ elements of $Q_8$, while $g_b$ and $g_c$ move only $4$ elements each; since conjugation preserves the number of moved points, we obtain $\beta \, g_a \, \beta^{-1} = g_a$ and $\beta \, \{g_b, \, g_c\} \, \beta^{-1} = \{g_b, \, g_c\}$. As $g_b$ and $g_c$ commute, it follows that $\beta \rho_i \beta^{-1} = \beta g_b \beta^{-1} \, \beta g_c \beta^{-1} = g_b g_c = \rho_i$. Therefore $\beta$ commutes with $\rho_i$ and $\rho_j$, and hence with all the right multiplications, since $Q_8$ is generated by $i$ and $j$. Then $\beta (z) = \beta (\rho_z (1)) = \rho_z (\beta (1)) = \beta(1) \, z$, for all $z \in Q_8$, i.e. $\beta = \lambda_{\beta(1)}$. Finally, $q \mapsto \sigma_q$ is an isomorphism of groups, since $\lambda_q \lambda_{q'} = \lambda_{qq'}$ and $B$ is the non-trivial coset of the subgroup $A$ of index $2$.
\end{example}

\section{Concluding remarks and open questions}\selabel{conclusion}

To summarize, we have shown that three classes of solutions which are rather ``small'' in the category $\YB$ of all YB-solutions (see the tables in \cite[Section 3]{MilBIP}) are universal and finitely universal. This raises several natural questions. Which other families of YB-solutions are universal, or finitely universal? Given such a family and a (finite) group $G$, how can we find the ``smallest'' YB-solution in it whose automorphism group is isomorphic to $G$, in the spirit of \cite{Babai1980, BarmakBarreto24}? We also recall the question raised in the Introduction:
\begin{center}
\emph{is there a functor $F \colon \UG \to \YB$ which is full and essentially surjective?\footnote{or even a full retraction of the embedding \equref{YBinUG}.}}
\end{center}
In other words, can every YB-solution, together with all its morphisms, be obtained from a graph? Of the three constructions given in this paper, only the assignment $\Gamma \mapsto (X_\Gamma, \, R_\Gamma)$ of \sseref{FS} is functorial with respect to all morphisms of graphs.

In \thref{field} we constructed, directly from the operations of a field $K$, a YB-solution whose automorphism group is ${\rm Aut} (K)$. This leads to the following question: given a ring $R$, can one construct, directly from the operations of $R$, a YB-solution whose automorphism group is isomorphic to ${\rm Aut} (R)$?

A related question is whether any monoid is isomorphic to the endomorphism monoid of a YB-solution; a positive answer would follow if $\YB$ were alg-universal in the sense of \cite{PultrTrnkova80}. We recall that every monoid is the endomorphism monoid of a set endowed with two maps, not necessarily commuting \cite[Chapter II, 5.7]{PultrTrnkova80}.

\subsection*{Acknowledgements}
The author thanks A.L. Agore and A. Chirvasitu for the many conversations of the last four years on the Yang--Baxter equation and related topics. The idea of associating YB-solutions to arbitrary graphs is due to the author. Claude Opus~5.5 (Anthropic) was used throughout the preparation of this paper, as follows: (1) to find references, improving the author’s (rather poor) English, to explore constructions and proofs, and for drafting and proofreading; (2) to write Python programs which check, for small sizes, that the maps in the examples are YB-solutions; the bipartite datum in \exref{Q8} was found by such a computer program. All statements, proofs and computational results were checked by the author, who takes full responsibility for the content of the paper.

\newcommand{\Addresses}{{%
  \bigskip
  \footnotesize
  \textsc{Faculty of Mathematics and Computer Science, University of Bucharest}
  \par\nopagebreak
  \textsc{Str. Academiei 14, RO-010014 Bucharest 1, Romania}
  \par\nopagebreak
  \textit{E-mail address}: \texttt{gigel.militaru@gmail.com}
}}

\Addresses

\end{document}